\documentclass{amsart}

\usepackage{amsmath,amssymb,amsthm}
\usepackage{a4wide}
\usepackage{hyperref}
\usepackage{array}

\usepackage{graphicx}

\newtheoremstyle{myremark} 
    {7pt}                    
    {7pt}                    
    {}  	                 
    {}                           
    {\bf}       	         
    {.}                          
    {.5em}                       
    {}  

\theoremstyle{plain}
\newtheorem{lemma}{Lemma}[section]
\newtheorem{theorem}[lemma]{Theorem}

\newtheorem{definition}[lemma]{Definition}
\newtheorem{corollary}[lemma]{Corollary}
\newtheorem{proposition}[lemma]{Proposition}

\theoremstyle{myremark}
\newtheorem{remark}[lemma]{Remark}
\newtheorem{example}[lemma]{Example}
\newtheorem{notation}[lemma]{Notation}

\newcommand{\zet}{\mathbb{Z}}

\newcommand{\ind}{\mathrm{Ind}}
\newcommand{\htpyequiv}{\simeq}
\newcommand{\conn}{\mathrm{conn}}
\newcommand{\incl}{\hookrightarrow}
\renewcommand{\subset}{\subseteq}
\newcommand{\susp}{\Sigma}

\newcommand{\add}[2]{\mathrm{Add}(#1;#2)}
\newcommand{\del}[2]{\mathrm{Del}(#1;#2)}

\newcommand{\tc}{\overline{C_n^r}}
\newcommand{\oth}{T_{3r+3}}

\begin{document}
\title{Splittings of independence complexes and the powers of cycles}

\author[Micha{\l} Adamaszek]{Micha{\l} Adamaszek}
\address{Mathematics Institute and DIMAP,
      \newline University of Warwick, Coventry, CV4 7AL, UK}
\email{aszek@mimuw.edu.pl}
\thanks{Research supported by the Centre for Discrete
        Mathematics and its Applications (DIMAP), EPSRC award EP/D063191/1.}


\begin{abstract}
We use two cofibre sequences to identify some combinatorial situations when the independence complex of a graph splits into a wedge sum of smaller independence complexes. Our main application is to give a recursive relation for the homotopy types of the independence complexes of powers of cycles, which answers an open question of D. Kozlov.
\end{abstract}
\date{\today.}
\maketitle

\section{Introduction}

For a finite simple graph $G$, its independence complex $\ind(G)$ is the simplicial complex whose vertices are the vertices of $G$ and whose simplices are the independent sets of $G$. It is a very well studied gadget in combinatorial algebraic topology. Here we investigate some combinatorial techniques for the problem of calculating the homotopy type of that complex for a given graph.

We study the complex $\ind(G)$ using the natural inclusions $\ind(G\setminus v)\incl\ind(G)$ and $\ind(G)\incl\ind(G- e)$ for a vertex $v$ and an edge $e$. They fit into two cofibre sequences
\begin{align}
\label{cofib1} \ind(G\setminus N[v])\incl\ind(G\setminus v)\incl\ind(G)\to\susp\,\ind(G\setminus N[v])\to\cdots, \\
\label{cofib2} \susp\,\ind(G\setminus N[e])\incl\ind(G)\incl\ind(G- e)\to\susp^2\,\ind(G\setminus N[e])\to\cdots.
\end{align}
Here $N[v]$ is the \emph{closed neighbourhood} of $v$ (which includes $v$ itself) and $N(v)=N[v]\setminus\{v\}$. For an edge $e=(u,v)$ we set $N[e]=N[u]\cup N[v]$. The notation $\susp\,K$ stands for the unreduced suspension of $K$.

Results based on various special instances of these sequences are scattered around in the literature, eg. \cite{Cso2,DochEng,EH,Eng1,Eng2,Kaw2,Kaw,Mar,Mesh}. For example the \emph{fold lemma} of \cite{Eng1}, which says that if $N(u)\subset N(v)$ then $\ind(G\setminus v)$ and $\ind(G)$ are homotopy equivalent, corresponds to the case where the first space in the cofibre sequence \eqref{cofib1} is contractible. Another interesting situation occurs when the map $\ind(G\setminus N[v])\incl\ind(G\setminus v)$ is null-homotopic, as then the cofibre sequence splits and we have an equivalence $\ind(G)\htpyequiv \ind(G\setminus v)\vee\susp\,\ind(G\setminus N[v])$. This happens, for example, when $N[u]\subset N[v]$ for some vertex $u$, as in \cite{Mar}.

In Section \ref{sect:first} we present a unified approach to results of this kind using \eqref{cofib1} and \eqref{cofib2}. We also identify combinatorial situations in which the two cofibre sequences lead to exact results. Another splitting result of Mayer-Vietoris type is analyzed in Section \ref{sect:mayer}. Section \ref{sect:unification} contains some applications and examples. In particular, we give quick proofs of some results of \cite{Eng1,Kaw,Mar,Thapper}.

We emphasize that the functorial behaviour of the independence complex under vertex removals and (contravariantly) under edge removals is our key technique. In particular, all homotopy equivalences and splittings we derive are natural, that is induced by some morphisms of the underlying graphs.

The main result of this paper comes in the last section where we use the splitting results associated with the sequence \eqref{cofib2} to calculate the homotopy types of independence complexes of a particular family of graphs, namely the powers $C_n^r$ of cycles. Recall that D. Kozlov in \cite{Koz} computed the homotopy types of $\ind(P_n)$ and $\ind(C_n)$, where $P_n$ is the path and $C_n$ is the cycle on $n$ vertices. The answers are determined by the homotopy equivalences
$$\ind(P_n)\htpyequiv \susp\,\ind(P_{n-3}),\quad \ind(C_n)\htpyequiv \susp\,\ind(C_{n-3}).$$
An open question of \cite{Koz} is to find similar statements for the complexes $\ind(P_n^r)$ and $\ind(C_n^r)$, $r\geq 2$. Here $G^r$ denotes the \emph{$r$-th distance power of $G$}, which is the graph with the same vertex set in which two vertices are adjacent if and only if their distance in $G$ is \emph{at most $r$}. Therefore $C_n^r$ is the graph spanned by the vertices of the $n$-gon, with two vertices being adjacent if and only if they are at most $r$ steps away along the perimeter of the $n$-gon. For $P_n^r$ the $n$-gon is replaced with an $n$-vertex path.

The answer for $\ind(P_n^r)$ is given in \cite{Eng1} in the form of a recursive relation
\footnote{Note that \cite{Koz,Eng1} denote our $\ind(P_n^r)$, $\ind(C_n^r)$ by, respectively, $\mathcal{L}_n^{r+1}$, $\mathcal{C}_n^{r+1}$.}
\begin{equation}
\label{engstrom-path}
\ind(P_n^r)\htpyequiv\susp\,\ind(P_{n-(r+2)}^r)\vee\susp\,\ind(P_{n-(r+3)}^r)\vee\cdots\vee\susp\,\ind(P_{n-(2r+1)}^r),\quad n\geq r+1.
\end{equation}
Here we obtain a corresponding statement for $\ind(C_n^r)$, answering the question raised in \cite{Koz,Eng1}.
\begin{theorem}
\label{thm:bigsplitting}
For every $r\geq 1$ and $n\geq 5r+4$ there is a homotopy equivalence
$$\ind(C_n^r)\htpyequiv\susp^2\,\ind(C_{n-(3r+3)}^r)\vee X_{n,r}$$
where $X_{n,r}$ is a space which splits, up to homotopy, into a wedge sum of complexes of the form $\susp^3\,\ind(P_{n-a}^r)$ for various values of $4r+6\leq a\leq 6r+3$.
\end{theorem}
The reader will see that the proof of the theorem gives an algorithmic way of enumerating all the wedge summands that go into $X_{n,r}$; there are asymptotically $r^3$ of them and we list them at the end of Section \ref{sect:allpowers}. For example, when $r=1$ we will have $\ind(C_n)\htpyequiv \susp^2\,\ind(C_{n-6})$ with $X_{n,1}$ being trivial, which agrees with Kozlov's recurrence. When $r=2$ the exact answer is
$$\ind(C_n^2)\htpyequiv \susp^2\,\ind(C_{n-9}^2)\,\vee\,\bigvee^4\susp^3\,\ind(P_{n-14}^2)\,\vee\,\bigvee^5\susp^3\,\ind(P_{n-15}^2),$$
and so on. 

The proof of Theorem \ref{thm:bigsplitting} can be found in Section \ref{sect:allpowers}. The idea is to find an explicit inclusion $\susp^2\,\ind(C_{n-(3r+3)}^r)\incl\ind(C_n^r)$ which splits off. It can also be seen as producing a quite unusual model of the space $\susp^2\,\ind(C_{n-(3r+3)}^r)$.

\section{Notation}
\label{sect:notation}

We first recall some notation. For a graph $G$ and subset $W\subset V(G)$ of the vertices let $G[W]$ denote the subgraph of $G$ induced by $W$ and let $\ind(G)[W]$ be the subcomplex of $\ind(G)$ induced by the vertex set $W$. We easily see
$$\ind(G)[W]=\ind(G[W])$$
and it follows that
$$\ind(G)[W]\cap \ind(G)[U] = \ind(G[W\cap U])$$
for any two vertex sets $W,U\subset V(G)$.
We are going to write $G\setminus v$ and $G\setminus W$ instead of the more correct $G[V(G)\setminus\{v\}]$ and $G[V(G)\setminus W]$. The notation $G- e$ or $G\cup e$ means $G$ with the edge $e$ removed or added. To avoid overloading curly brackets $\{\cdot\}$, edges will be denoted by $e=(u,v)$, which should not suggest that they are directed. By $N_G(u)$ and $N_G[u]$ we mean the open and closed neighbourhood of $u$ in $G$ and we write $N[u]$ and $N(u)$ when there is no danger of ambiguity. If $e=(u,v)$ is an edge in $G$ we define the \emph{closed neighbourhood} of $e$ as $N[e]=N[u]\cup N[v]$.

The symbols $P_n$, $C_n$ and $K_n$ denote the path, cycle and complete graph with $n$ vertices. They are understood to be the empty graph when $n\leq 0$. 

If $G\sqcup H$ is the disjoint union of two graphs then its independence complex satisfies
$$\ind(G\sqcup H)=\ind(G)\ast\ind(H)$$
where $\ast$ is the simplicial join. In particular, if $\ind(G)$ is contractible then so is $\ind(G\sqcup H)$ for any $H$. If $e$ is understood as the graph consisting of a single edge then $\ind(e)=S^0$ and $\ind(e\sqcup G)=S^0\ast \ind(G)=\susp\,\ind(G)$ is the suspension of $\ind(G)$. 

Many results can be nicely phrased in the language of cofibre sequences. For any continuous map $f:A\to X$ the \emph{homotopy cofibre} (or mapping cone) is the space
$$C(f)=(X\sqcup (A\times [0,1]))/f(a)\sim (a,1),\ (a,0)\sim(a',0).$$
If $f:A\incl X$ is a subcomplex inclusion then $C(f)$ is just $X$ with a cone over $A$ attached and it is homotopy equivalent to $X/A$. There is a cofibre (or Puppe) sequence
$$A\xrightarrow{f} X\incl C(f) \to \susp\,A\xrightarrow{\susp f} \susp\,X\to\susp\,C(f)\to\susp^2\,A\to\cdots$$
with the property that every consecutive triple is, up to homotopy, a map followed by its mapping cone. Since the homotopy type of $C(f)$ depends only on the homotopy class of $f$, we get that if $f:A\to X$ is null-homotopic then $C(f)\htpyequiv X\vee\susp\,A$. In particular, if $A$ is contractible then $C(f)\htpyequiv X$.

We refer to \cite{Book} for facts about (combinatorial) algebraic topology.

\section{Two cofibre sequences and their consequences}
\label{sect:first}

We start with vertex removals. Various parts of the next proposition are well-known.
\begin{proposition}
\label{prop:les-v}
There is always a cofibre sequence
$$\ind(G\setminus N[v])\incl\ind(G\setminus v)\incl\ind(G)\to\susp\,\ind(G\setminus N[v])\to\cdots.$$
In particular
\begin{itemize}
\item[a)] if $\ind(G\setminus N[v])$ is contractible then the natural inclusion $\ind(G\setminus v)\incl\ind(G)$ is a homotopy equivalence,
\item[b)] if the map $\ind(G\setminus N[v])\incl\ind(G\setminus v)$ is null-homotopic then there is a splitting
$$\ind(G)\htpyequiv \ind(G\setminus v) \vee \susp\,\ind(G\setminus N[v]).$$
\end{itemize}
\end{proposition}
\begin{proof}
Any independent set in $G$ is either contained in $G\setminus v$ or it is the union of $\{v\}$ and some independent set in $G\setminus N[v]$, so we have a decomposition
$$\ind(G)=S\cup T$$
where
$$S=\ind(G\setminus v),\quad T=v\ast\ind(G\setminus N[v])\htpyequiv\ast, \quad S\cap T=\ind(G\setminus N[v]).$$
Therefore $\ind(G)$ is the homotopy cofibre of the inclusion $S\cap T\incl S$. The statements a) and b) follow from the properties discussed in Section \ref{sect:notation}.
\end{proof}
The ``generic combinatorial cases'' of a) and b) are the following.
\begin{theorem}[\cite{Eng1}]
\label{smallthm}
If $u,v$ are two distinct vertices with $N(u)\subset N(v)$ then there is a homotopy equivalence
$$\ind(G)\htpyequiv\ind(G\setminus v).$$
\end{theorem} 
\begin{theorem}[\cite{Mar}]
\label{closedneib}
If $u,v$ are two distinct vertices with $N[u]\subset N[v]$ then there is a homotopy equivalence
$$\ind(G)\htpyequiv \ind(G\setminus v) \vee \susp\,\ind(G\setminus N[v]).$$
\end{theorem}
\begin{proof}[Proof of Theorems \ref{smallthm} and \ref{closedneib}.]
If $u$ is such that $N(u)\subset N(v)$ then the graph $G\setminus N[v]$ has $u$ as an isolated vertex, hence the complex $\ind(G\setminus N[v])$ is contractible and Theorem \ref{smallthm} follows from part a) above. If, on the other hand, $u$ is such that $N[u]\subset N[v]$ then the inclusion $\ind(G\setminus N[v])\incl\ind(G\setminus v)$ factors through the contractible space $u\ast\ind(G\setminus N[v])$, so Theorem \ref{closedneib} follows from part b).
\end{proof}
Note that the condition $N(u)\subset N(v)$ implies that $u$ and $v$ are not adjacent in $G$, while $N[u]\subset N[v]$ forces them to be adjacent.

A similar discussion applies to edges. If $e=(u,v)$ is an edge then $e\sqcup(G\setminus N[e])$ is the induced subgraph of $G$ whose vertices are $u$, $v$ and all the vertices of $G\setminus N[e]$. Then we have the following proposition.
\begin{proposition}
\label{prop:les-e}
There is always a cofibre sequence
$$\ind(e\sqcup(G\setminus N[e]))\incl\ind(G)\incl\ind(G- e)\to\susp\,\ind(e\sqcup(G\setminus N[e]))\to\cdots.$$
In particular
\begin{itemize}
\item[a)] if $\ind(G\setminus N[e])$ is contractible then the natural inclusion $\ind(G)\incl\ind(G- e)$ is a homotopy equivalence,
\item[b)] if the map $\ind(e\sqcup(G\setminus N[e]))\incl\ind(G)$ is null-homotopic then there is a splitting
$$\ind(G- e)\htpyequiv \ind(G) \vee \susp^2\,\ind(G\setminus N[e]).$$
\end{itemize}
\end{proposition}
\begin{proof}
The first statement is an observation of \cite{Mesh}: any independent set in $G- e$ is either independent in $G$ or it contains both endpoints of $e$ together with some independent set in $G\setminus N[e]$. This gives a decomposition
$$\ind(G- e)=K\cup L$$
where
$$K=\ind(G),\quad L=e\ast\ind(G\setminus N[e])\htpyequiv\ast, \quad K\cap L=\ind(e\sqcup(G\setminus N[e])).$$
Again, it means that $\ind(G- e)$ is homotopy equivalent to the homotopy cofibre of the inclusion $K\cap L\incl K$. The statements a) and b) follow from the properties discussed in Section \ref{sect:notation} and the fact that $\ind(e\sqcup(G\setminus N[e]))=\susp\,\ind(G\setminus N[e])$.
\end{proof}

As before there are some useful special circumstances when conditions a) and b) can be verified at the combinatorial level.
\begin{definition}
\label{def:isolating}
An edge $e=(u,v)$ in $G$ is called \emph{isolating} if the induced subgraph $G\setminus N[e]$ has an isolated vertex.
\end{definition}
Clearly part a) holds for isolating edges, i.e. the removal of an isolating edge does not change the homotopy type of the independence complex. Note that any such statement can also be used in the opposite direction, that is to say that the insertion of an edge which becomes isolating preserves the homotopy type. 

The situations where part b) of Proposition \ref{prop:les-e} applies are more complicated.
\begin{theorem}
\label{general-p4}
Let $e=(u,v)$ be an edge in $G$. Suppose $T\subset G$ is an induced subgraph which contains the edge $e$ and such that $\ind(T)$ is contractible and, moreover, for every $x\in T$ we have $N[x]\subset N[e]$. Then the inclusion $\ind(e\sqcup(G\setminus N[e]))\incl\ind(G)$ is null-homotopic. Consequently, there is a splitting
$$\ind(G- e)\htpyequiv \ind(G) \vee \susp^2\,\ind(G\setminus N[e]).$$
\end{theorem}
\begin{proof}
The inclusion $\ind(e\sqcup(G\setminus N[e]))\incl\ind(G)$ factors through $\ind(G[V(T)\cup(V(G)\setminus N[e])])$. Since neither of the vertices $x\in V(T)$ has an edge to $V(G)\setminus N[e]$, the last graph is in fact $T\sqcup (G\setminus N[e])$, so its independence complex is a join where one of the factors is $\ind(T)\htpyequiv\ast$. It means that our inclusion factors through a contractible space.
\end{proof}
The simplest graph which can play the role of $T$ in the last statement is the $4$-vertex path $P_4$, hence we have the next corollary, which will be one of the main tools in Section \ref{sect:allpowers}. 
\begin{theorem}
\label{p4}
Let $e=(u,v)$ be an edge in $G$. Suppose there are vertices $x,y\in N[e]$ such that $N[x]\cup N[y]\subset N[e]$ and the induced subgraph $G[x,y,u,v]$ is isomorphic to the $4$-vertex path $P_4$. Then the inclusion $\ind(e\sqcup(G\setminus N[e]))\incl\ind(G)$ is null-homotopic. 
Consequently, there is a splitting
$$\ind(G- e)\htpyequiv \ind(G) \vee \susp^2\,\ind(G\setminus N[e]).$$
\end{theorem}

\begin{figure}
\begin{tabular}{m{2cm}m{0.4cm}m{2cm}m{0.4cm}m{2cm}m{0.4cm}m{2cm}}
\includegraphics[scale=0.5]{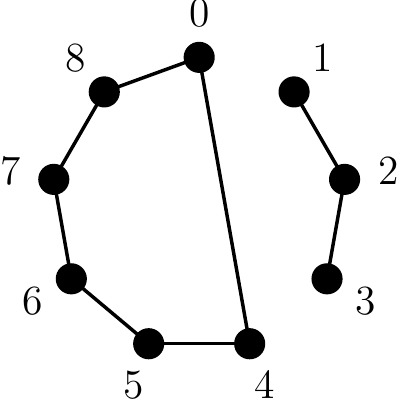} & $\hookrightarrow$ & \includegraphics[scale=0.5]{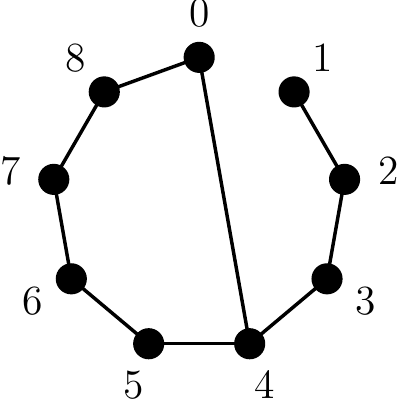} & $\hookrightarrow$ & \includegraphics[scale=0.5]{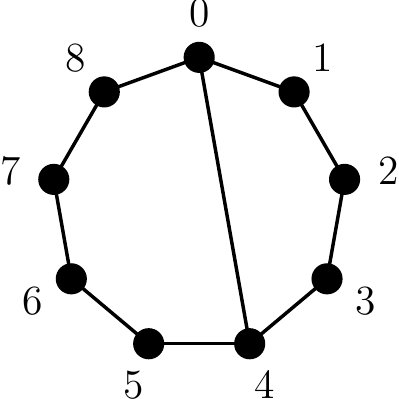} & $\hookleftarrow$ & \includegraphics[scale=0.5]{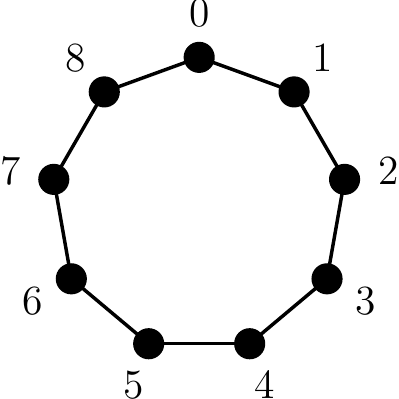}
\end{tabular}
\caption{Graph inclusions in Example \ref{exkozlov} which induce a zigzag of homotopy equivalences upon the application of $\ind(\cdot)$.}
\label{fig:cycles}
\end{figure}

\begin{example}\label{exkozlov}
 We illustrate the applications of isolating edges by reproving Kozlov's equivalence $\ind(C_n)\htpyequiv\susp\,\ind(C_{n-3})$. We present the argument in detail as it is the prototype of the methods used in Section \ref{sect:allpowers}. See Fig.\ref{fig:cycles}.

Start with the cycle $C_n$ with vertices labeled $0,\ldots,n-1$. Let $C_n'=C_n\cup\{(0,4)\}$. The edge $(0,4)$ in $C_n'$ is isolating because removing $N[0]\cup N[4]$ leaves $2$ isolated. By Proposition \ref{prop:les-e}.a) it means that extending $C_n$ to $C_n'$ preserves the homotopy type of the independence complex. More  precisely, the induced inclusion
$$\ind(C_n')\incl\ind(C_n)$$
is an equivalence. Now in $C_n'$ the edge $(0,1)$ is isolating as removing $N[0]\cup N[1]$ isolates $3$. We can delete $(0,1)$ without affecting the independence complex (up to homotopy). Then in $C_n'\setminus\{(0,1)\}$ the edge $(3,4)$ is isolating as removing $N[3]\cup N[4]$ isolates $1$. Again, we can delete $(3,4)$. But the graph we finally obtained, $C_n'\setminus\{(0,1),(3,4)\}$, is a disjoint union of a path $1-2-3$ and $C_{n-3}$ so its independence complex is homotopy equivalent to $S^0\ast\ind(C_{n-3})=\susp\,\ind(C_{n-3})$. We obtain a zigzag of equivalences
$$\susp\,\ind(C_{n-3})\htpyequiv\ind(C_n'\setminus\{(0,1),(3,4)\})\xleftarrow{\htpyequiv}\ind(C_n'\setminus\{(0,1)\})\xleftarrow{\htpyequiv}\ind(C_n')\xrightarrow{\htpyequiv}\ind(C_n)$$
in which every map is induced functorially by some graph morphism.
\end{example}

\begin{notation} From now on we are going to abbreviate such arguments by writing: there is a sequence of isolating operations
$$\add{0,4}{2},\,\del{0,1}{3},\,\del{3,4}{1}$$
which reads: add the edge $(0,4)$, where $2$ is the vertex that certifies the isolating property, then remove $(0,1)$ for which $3$ is the certificate etc. Note that such sequence of operations is indeed a \emph{sequence}: they may no longer be isolating if performed in a different order. Every isolating sequence generates a zigzag of weak equivalences as in the example.
\end{notation}

\section{Mayer-Vietoris splitting}
\label{sect:mayer}

Combinatorial splittings can also be obtained from the following result.
\begin{theorem}
\label{bigthm}
Suppose $X,Y\subset V(G)$ are two vertex sets which satisfy the conditions:
\begin{itemize}
\item $X\cup Y=V(G)$,
\item the independence complex of $G[X\cap Y]$ is contractible,
\item every vertex in $X\setminus Y$ has an edge to every vertex of $Y\setminus X$.
\end{itemize}
Then there is a splitting
$$
\ind(G)\htpyequiv \ind(G[X])\vee \ind(G[Y])
$$
which is natural in the sense that the inclusions $\ind(G[X])\incl\ind(G)$ and $\ind(G[Y])\incl\ind(G)$ induced by inclusions of $G[X]$ and $G[Y]$ in $G$ are homotopic to the inclusions of the two wedge summands.
\end{theorem}
\begin{proof}
Let $K=\ind(G[X])$ and $L=\ind(G[Y])$. First let us check that $K\cup L=\ind(G)$. Suppose $\sigma$ is an independent set in $G$ and $\sigma\not\in L$. Then $\sigma$ must have a vertex $v$ in $X\setminus Y$. The third condition implies that $\sigma$ cannot have any vertices in $Y\setminus X$, therefore $\sigma\subset X$ which means $\sigma\in K$. That completes the verification.

Now $\ind(G)$ is the union $K\cup L$ of two subcomplexes such that $K\cap L=\ind(G[X\cap Y])$ is contractible. Then there is an equivalence $\ind(G)\htpyequiv K\vee L$. 
\end{proof}

We can use it to identify the graph inclusions corresponding to the two summands in Theorem \ref{closedneib} and Theorem \ref{general-p4}.

\begin{proof}[Proof of Theorem \ref{closedneib} from Theorem \ref{bigthm}.]
We use the previous theorem with $X=(V(G)\setminus N(v))\cup \{u\}$ and $Y=V(G)\setminus\{v\}$. Clearly $X\cup Y=V(G)$. Since $X\cap Y=\{u\}\cup(V(G)\setminus N[v])$ and $u$ does not have any edges to $V(G)\setminus N[v]$, the induced graph $G[X\cap Y]$ has $u$ as an isolated vertex, so the complex $\ind(G[X\cap Y])$ is contractible. Finally $X\setminus Y=\{v\}$ and $Y\setminus X=N(v)$ so the third condition in Theorem \ref{bigthm} is automatically satisfied. 

In the splitting obtained from Theorem \ref{bigthm} the complex $\ind(G[Y])$ is $\ind(G\setminus v)$. The graph in the second summand, $G[X]$, is the disjoint union of an edge $e=(u,v)$ with $G\setminus N[v]$. This is because neither $v$ (by definition) nor $u$ (by assumption) have edges to $V(G)\setminus N[v]$. It follows that $\ind(G[X])=\ind(e)\ast\ind(G\setminus N[v])=S^0\ast\ind(G\setminus N[v])=\susp\,\ind(G\setminus N[v])$ and the proof is complete.
\end{proof}

\begin{proof}[Proof of Theorem \ref{general-p4} from Theorem \ref{bigthm}.]
Let $H=G- e$. First extend $H$ to a bigger graph $H_+$ by adding an extra vertex $w$ with edges to $N(u)\cup N(v)$. The inclusion $\ind(H)\incl\ind(H_+)$ is an homotopy equivalence by Proposition \ref{prop:les-v}.a) (because $H_+\setminus N[w]$ contains isolated vertices $u$, $v$). In $H_+$ the operation $\add{u,v}{w}$ is isolating. Let $G_+$ denote the resulting graph. It contains $G$ as $G_+\setminus w$ and $\ind(G_+)\htpyequiv \ind(G- e)$.

Set $X=(V(G)\setminus N[e])\cup V(T)\cup \{w\}$ and $Y=V(G)$. Clearly $X\cup Y=V(G_+)$. Since $X\cap Y=V(T)\cup(V(G)\setminus N[e])$ and $V(T)$ has no edges to $V(G)\setminus N[e]$, the induced graph $G[X\cap Y]$ contains $T$ as a connected component and $\ind(G[X\cap Y])$ is contractible. Finally $X\setminus Y=\{w\}$ and $Y\setminus X\subset N(w)$ so the third condition in Theorem \ref{bigthm} is automatically satisfied. 

In the splitting of $\ind(G_+)$ obtained from Theorem \ref{bigthm} the complex $\ind(G_+[Y])$ is $\ind(G)$. The graph in the other summand, $G_+[X]$, is the disjoint union of $G_+[V(T)\cup\{w\}]$ with $G\setminus N[e]$. But $\ind(G_+[V(T)\cup\{w\}])$ consists of the contractible subspace $\ind(T)$ together with two edges $wu$ and $wv$, so it is homotopy equivalent to $S^1$. It follows that $\ind(G_+[X])\htpyequiv S^1\ast\ind(G\setminus N[e])=\susp^2\,\ind(G\setminus N[e])$ and the proof is complete.
\end{proof}

\section{Applications and examples}
\label{sect:unification}
We start with a simple application of isolating edges to a known reduction result.
\begin{lemma}[\cite{Cso2,Mar,Bar}]
\label{3-path}
Let $G$ be a graph and $e=(x,y)$ an edge. If $G'$ is obtained from $G$ by replacing $e$ with a path $x-u-v-w-y$ with $3$ new vertices then $\ind(G')\htpyequiv\susp\,\ind(G)$.
\end{lemma}
\begin{proof}
There is a sequence of isolating operations in $G'$:
$$\add{x,y}{v},\,\del{x,u}{w},\,\del{y,w}{u}$$
which results in the graph $\{(u,v),(v,w)\}\sqcup G$.
\end{proof}
A more general result we can recover using isolating operations follows also from the main theorem of \cite{Bar}.
\begin{lemma}
\label{deg2}
If $v$ is a vertex of $G$ of degree $2$ with neighbours $u$, $w$ which satisfy $N[u]\cap N[w]=\{v\}$ then $\ind(G)$ is homotopy equivalent to $\susp\,\ind(G')$ where $G'$ is obtained from $G$ by removing $u, v, w$ and spanning a complete bipartite graph between vertices which belonged to $N[u]$ and those from $N[w]$.
\end{lemma}
\begin{proof}
Denote $U=N[u]\setminus\{u,v\}$ and $W=N[w]\setminus \{w,v\}$. We can first perform all isolating insertions $\add{x,y}{v}$ for all pairs $x\in U,\, y\in W$ which had not already been an edge. Then we can perform isolating deletions $\del{u,x}{w}$ for $x\in U$ followed by $\del{w,y}{u}$ for $y\in W$. We end up with $\{(u,v),(v,w)\}\sqcup G'$ and conclude as before. 
\end{proof}
Let us mention two more specializations of Theorem \ref{closedneib}.
\begin{corollary}
\label{deg1}
If $u$ is a vertex of degree $1$ and $v$ is its only neighbour then $\ind(G)\htpyequiv\susp\,\ind(G\setminus N[v])$.
\end{corollary}
\begin{proof}
The vertices $u$ and $v$ satisfy the assumptions of Theorem \ref{closedneib}. Moreover $G\setminus v$ has $u$ as an isolated vertex so $\ind(G\setminus v)$ is contractible.
\end{proof}
\begin{corollary}[\cite{Eng1,Eng2,Kaw}]
\label{clique-neib}
Let $u$ be a vertex such that $N(u)$ is a clique. Then there is a homotopy equivalence
$$\ind(G)\htpyequiv \bigvee_{v\in N(u)} \susp\, \ind(G\setminus N[v]).$$
\end{corollary}
\begin{proof}
Let $v\in N(u)$ be any vertex. Then $N[u]\subset N[v]$ so $\ind(G)$ splits into $\susp\,\ind(G\setminus N[v])$ and $\ind(G\setminus v)$. Let $G'=G\setminus v$. In $G'$ the neighbours of $u$ again form a clique so by induction $\ind(G')$ splits as $\bigvee_{v'\in N_{G'}(u)}\susp\,\ind(G'\setminus N_{G'}[v'])$. However, since $(v,v')$ is an edge in $G$ for all $v'\in N_{G'}(u)$ we have $G'\setminus N_{G'}[v']=G\setminus N[v']$ which together with the first summand gives the desired splitting.
\end{proof}

\begin{example}
Suppose $G$ is a connected graph with $n$ vertices and $m$ edges. Let $G_3$ denote the graph obtained from $G$ by subdividing each edge into $3$ parts. Let $e$ be any of the ``middle'' edges of $G_3$, that is edges connecting two subdividing vertices. 

In $G_3-e$ we have two vertices of degree $1$ and we see that it can be reduced to the empty graph by successfully applying Corollary \ref{deg1} $n$ times, once for each vertex of $G$. In $G_3\setminus N[e]$ the situation is similar, but this time we perform one reduction for each of the remaining $m-1$ edges. It means that $\ind(G_3-e)\htpyequiv \susp^n\emptyset=S^{n-1}$ and $\ind(G_3\setminus N[e])\htpyequiv\susp^{m-1}\emptyset=S^{m-2}$ and the cofibre sequence \eqref{cofib2} becomes
$$S^{m-1}\to\ind(G_3)\to S^{n-1}\to S^m\to \susp\,\ind(G_3)\to S^n\to S^{m+1}\to\cdots.$$
If $G$ is not a tree then $m>n-1$ so the map $S^{n-1}\to S^m$ must be null-homotopic and we get $\susp\,\ind(G_3)\htpyequiv S^n\vee S^m$. This almost recovers the result of Csorba \cite{Cso2} who proved that in fact $\ind(G_3)\htpyequiv S^{n-1}\vee S^{m-1}$.
\end{example}

Before stating the next result recall that the \emph{domination number} $\gamma(G)$ of $G$ is the minimal cardinality of a \emph{dominating set} in $G$, that is a subset $W\subset V(G)$ such that $\bigcup_{w\in W}N[w]=V(G)$. A graph is \emph{chordal} if it does not have an induced cycle of length at least $4$. Moreover, let $\psi$ be a function on graphs with values in $\{0,1,\ldots\}\cup\{\infty\}$ defined as follows:
\begin{displaymath}
\psi(G)=\begin{cases}
0 & \textrm{ if } G=\emptyset\\
\infty & \textrm{ if } G\neq\emptyset \textrm{ is edgeless }\\
\max_{e\in E(G)}\big\{\min \{\psi(G- e),\,\psi(G\setminus N[e])+1\}\big\} & \textrm{ otherwise. }
\end{cases}
\end{displaymath}
A graph is said to \emph{satisfy the Aharoni-Berger-Ziv} conjecture if $\psi(G)=\conn(\ind(G))+2$, where $\conn(\cdot)$ denotes the topological connectivity of a space. It is known that always $\psi(G)\leq\conn(\ind(G))+2$ (see \cite{AdaB,ABZ,Mesh}) and that there are examples when the inequality is strict \cite{AdaB}. The following was proved in \cite{Kaw}, with the ``wedge of spheres'' part also following from earlier results.
\begin{corollary}
Suppose $G$ is a chordal graph.
\begin{itemize}
\item[a)] \cite{VanTuyl,Russ,DochEng,Kaw} $\ind(G)$ is either contractible or homotopy equivalent to a wedge of spheres of dimension at least $\gamma(G)-1$,
\item[b)] \cite{Kaw} $G$ satisfies the Aharoni-Berger-Ziv conjecture.
\end{itemize}
\end{corollary}
\begin{proof}
a) The result is true for the empty graph (we assume $S^{-1}=\emptyset$, which is consistent with $\susp S^{-1}=S^0$) and for any discrete graph. Now suppose $G$ has at least one edge. By a well-known characterization (see \cite[Thm. 9.21]{BonMur}) every chordal graph has a vertex $u$ such that $N(u)$ is a clique. Choose any $v\in N(u)$. Then $N[u]\subset N[v]$, so by Theorem \ref{closedneib}
$$\ind(G)\htpyequiv\ind(G\setminus v)\vee\susp\,\ind(G\setminus N[v]).$$

Both graphs $G\setminus v$ and $G\setminus N[v]$ are chordal so by induction their independence complexes are either contractible or equivalent to wedges of spheres of dimension at least, respectively, $\gamma(G \setminus v)-1$ and $\gamma(G\setminus N[v])-1$. Of course $\gamma(G)\leq\gamma(G\setminus N[v])+1$. Moreover, every dominating set in $G\setminus v$ is also dominating in $G$ because to dominate $u$ it must contain a vertex in $N[u]\setminus\{v\}\subset N[v]$. It means that $\gamma(G)\leq\gamma(G\setminus v)$. It follows that each wedge summand of $\ind(G)$ is either contractible or a wedge of spheres of dimension at least $\min\{\gamma(G)-1,(\gamma(G)-2)+1\}=\gamma(G)-1$.

b) Let $f=(u,v)$. Then in the graph $G- f$ we have  $N(u)\subset N(v)$, so by Theorem \ref{smallthm} the complex $\ind(G- f)$ is homotopy equivalent to $\ind(G\setminus v)$. The complex $\ind(G\setminus N[v])$ is clearly equal to $\ind(G\setminus N[f])$ as $N[u]\subset N[v]$ in $G$. The splitting of a) can thus be rewritten as
$$\ind(G)\htpyequiv \ind(G- f)\vee\susp\,\ind(G\setminus N[f]).$$
The graph $G\setminus N[f]$ is chordal and a quick verification shows that the condition $N[u]\subset N[v]$ and the fact that $N(u)$ is a clique imply that also $G- f$ is chordal. By a) their independence complexes are wedges of spheres, so we have
\begin{align*}
\conn(\ind(G))+2&=\min\{\conn(\ind(G- f)),\,\conn(\ind(G\setminus N[f]))+1\}+2\\
&=\min\{\psi(G- f),\,\psi(G\setminus N[f])+1\}\\
&\leq\max_{e\in E(G)}\big\{\min\{\psi(G- e),\,\psi(G\setminus N[e])+1\}\big\}=\psi(G)
\end{align*}
where the first equality follows from the splitting. Since we always have $\psi(G)\leq\conn(\ind(G))+2$, we get $\psi(G)=\conn(\ind(G))+2$.
\end{proof}

\begin{figure}
\includegraphics[scale=0.8]{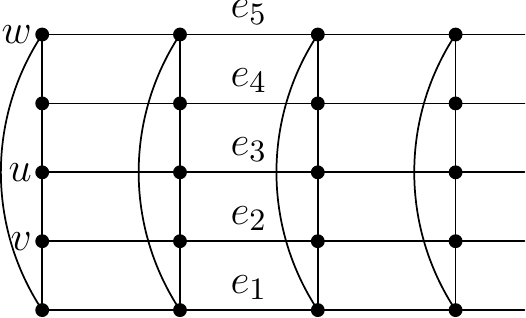} 
\caption{The graph $P_n\times C_5$.}
\label{fig:pnc5}
\end{figure}

\begin{example}
The next example is related to the independence complexes of cylindrical grids and the hard-squares model in statistical physics, as in \cite{JonCylinder}. Let $G=P_n\times C_5$ be the graph of Fig.\ref{fig:pnc5}. Let us show that $\ind(G\setminus N[e_1])\htpyequiv \ast$. In the graph $G\setminus N[e_1]$ we can apply Corollary \ref{deg1} to $v$, which is of degree $1$ with unique neighbour $u$. However, removing $N[u]$ leaves $w$ isolated, so the claim is proved. By Proposition \ref{prop:les-e}.a) we get $\ind(G)\htpyequiv \ind(G-e_1)$. We can remove $e_2,\ldots,e_5$ in the same way and finally
$$\ind(P_n\times C_5)\htpyequiv \ind(P_{n-2}\times C_{5})\ast\ind(P_2\times C_5)\htpyequiv\ind(P_{n-2}\times C_{5})\ast S^1= \susp^2\,\ind(P_{n-2}\times C_5)$$
where the equivalence $\ind(P_2\times C_5)\htpyequiv S^1$ is left to the reader. This was also found in \cite{Thapper} using explicit Morse matchings. Similar results can be obtained with the same method for other small grids.
\end{example}

\section{Powers of cycles}
\label{sect:allpowers}
In this section we develop a systematic approach to the complexes $\ind(C_n^r)$ and prove Theorem \ref{thm:bigsplitting}. The idea of the proof is as follows. We extend $C_n^r$ to another graph $\tc$ on the same vertex set but with more edges. The new graph will have the property that $\ind(\tc)\htpyequiv\susp^2\,\ind(C_{n-(3r+3)}^r)$ (Proposition \ref{prop:model}). Since $\tc$ is obtained from $C_n^r$ by inserting new edges, we get a natural inclusion
$$\susp^2\,\ind(C_{n-(3r+3)}^r)\htpyequiv\ind(\tc)\incl\ind(C_n^r).$$
This is our guess for what the inclusion of the first wedge summand in Theorem \ref{thm:bigsplitting} should be. We then need to show that, up to homotopy, the image of this inclusion indeed splits off. This is accomplished by analyzing the construction of $\tc$ from $C_n^r$ edge by edge and showing that every single edge insertion yields a splittable inclusion of independence complexes. For  this we use an obvious inductive consequence of Proposition \ref{prop:les-e}.b), which we record below for convenience.

\begin{lemma}
\label{lem:inductive}
Suppose $G$ is a graph and $e_1,\ldots,e_k$ is a sequence of edges which are \emph{not} in $G$. Let $G_0=G$ and let $G_i=G_{i-1}\cup e_i$ for $1\leq i\leq k$. Suppose that for each $i=1,\ldots,k$ the inclusion
$$\ind(e_i\sqcup(G_i\setminus N[e_i]))\incl\ind(G_i)$$
is null-homotopic. Then there is a homotopy equivalence
$$\ind(G)\htpyequiv \ind(G_k)\vee\bigvee_{i=1}^k \susp^2\,\ind(G_i\setminus N[e_i]).$$
\end{lemma}
\begin{proof}
For every $i=1,\ldots,k$ we have $G_i- e_i=G_{i-1}$, so Proposition \ref{prop:les-e}.b) yields splittings $\ind(G_{i-1})\htpyequiv \ind(G_i)\vee\susp^2\,\ind(G_i\setminus N[e_i])$, from which the result follows by induction.
\end{proof}

We now describe the construction of the graph $\tc$. The vertices of an $n$-cycle are labeled with elements of $\zet/n$. We start with $C_n^r$ and add new edges in the order described below (see Fig.\ref{fig:stages} and Fig.\ref{fig:general}).
\begin{itemize}
\item \emph{First phase.} It consists of $r-1$ stages.
\begin{itemize}
\item In stage $s$, where $1\leq s \leq r-1$, we add two groups of edges:
\begin{itemize}
\item first group: $(i,i+2r-s+2)$ for $i=1,\ldots,r+s+1$,
\item second group: $(i,i+3r-s+3)$ for $i=1,\ldots,s$.
\end{itemize}
\end{itemize}
\item \emph{Second phase.} Add all the edges of the form
$$(-x,3r+3+y)\quad\mathrm{for}\quad 0\leq x\leq r-1,\, 1\leq y\leq r,\, x+y\leq r.$$
\end{itemize}

\begin{figure}
\begin{tabular}{cc}
\includegraphics[scale=0.7]{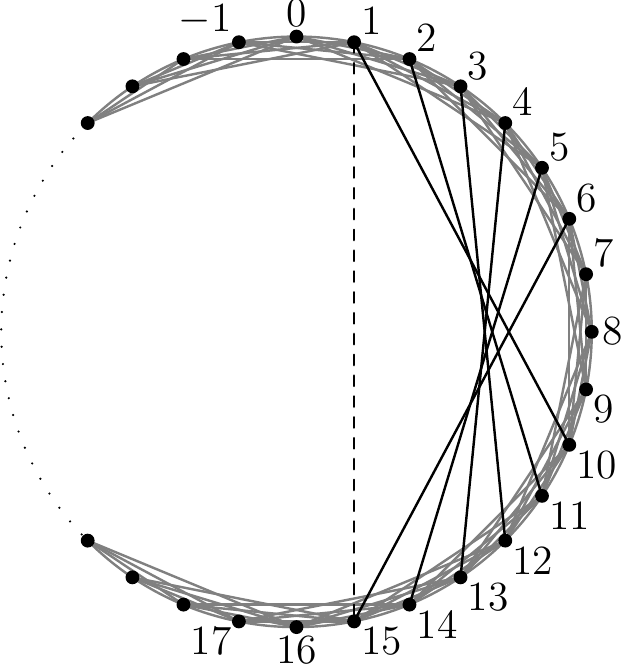} & \includegraphics[scale=0.7]{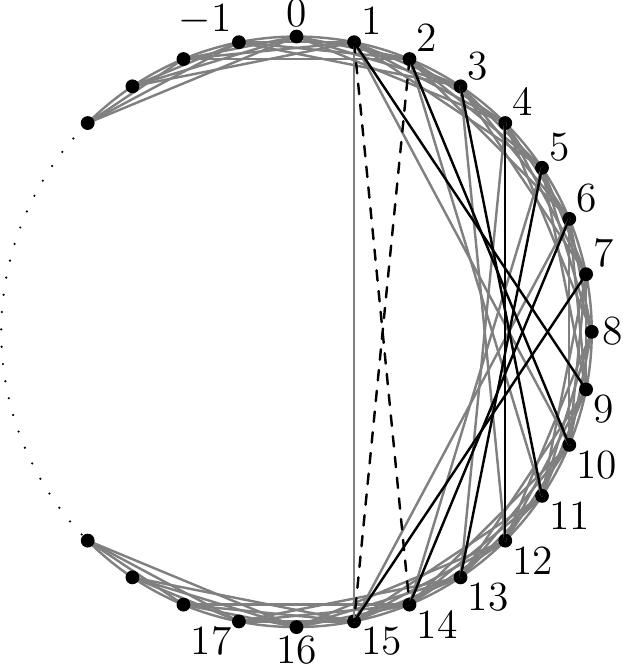}\\
a) & b) \vspace{0.3cm}\\ 
\includegraphics[scale=0.7]{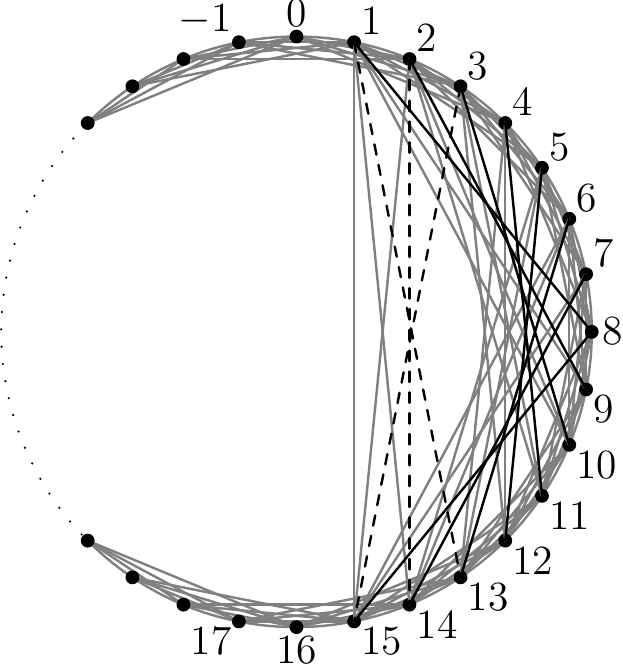} & \includegraphics[scale=0.7]{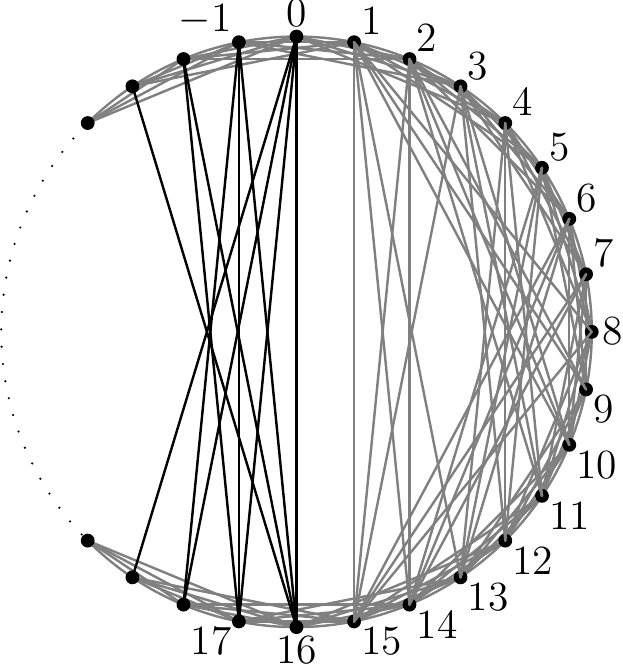}\\
c) & d) 
\end{tabular}
\caption{Construction of $\tc$ for $r=4$. Figures a),b),c) highlight edges added in stages $s=1,2,3$, where the edges of the first group are solid and those of the second group are dashed. Figure d) highlights the edges of the second phase.}
\label{fig:stages}
\end{figure}

\begin{figure}
\begin{tabular}{ccc}
\includegraphics[scale=0.7]{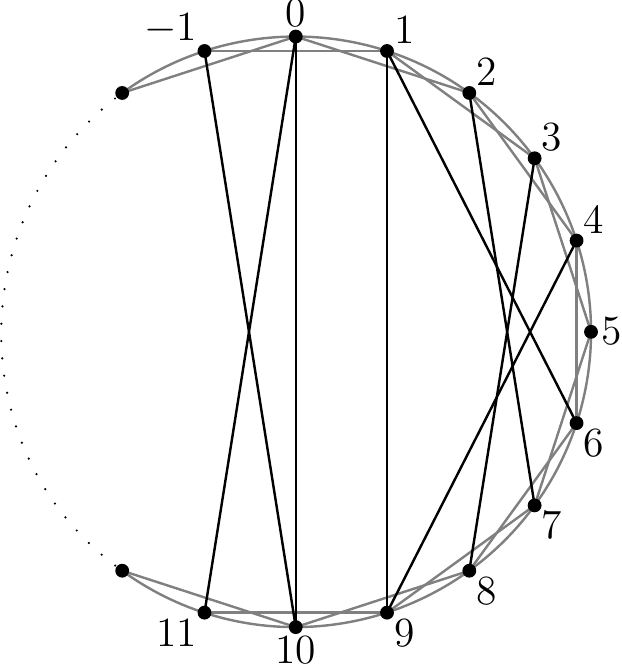} & \includegraphics[scale=0.7]{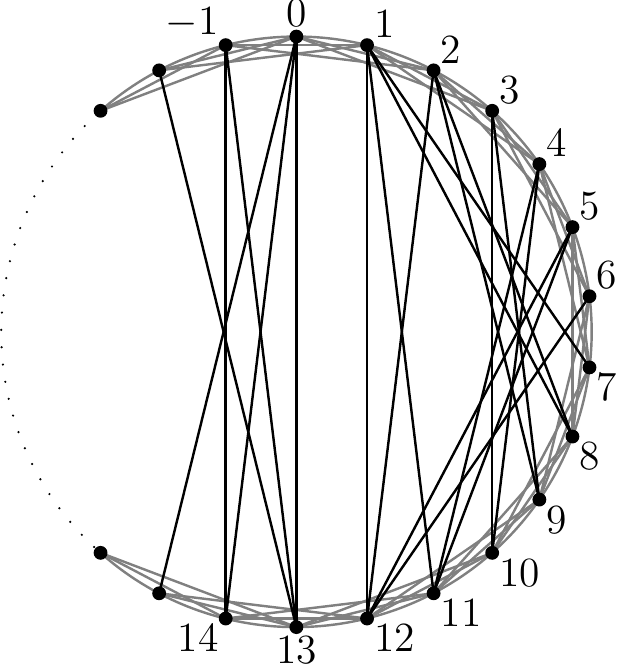} 
\includegraphics[scale=0.7]{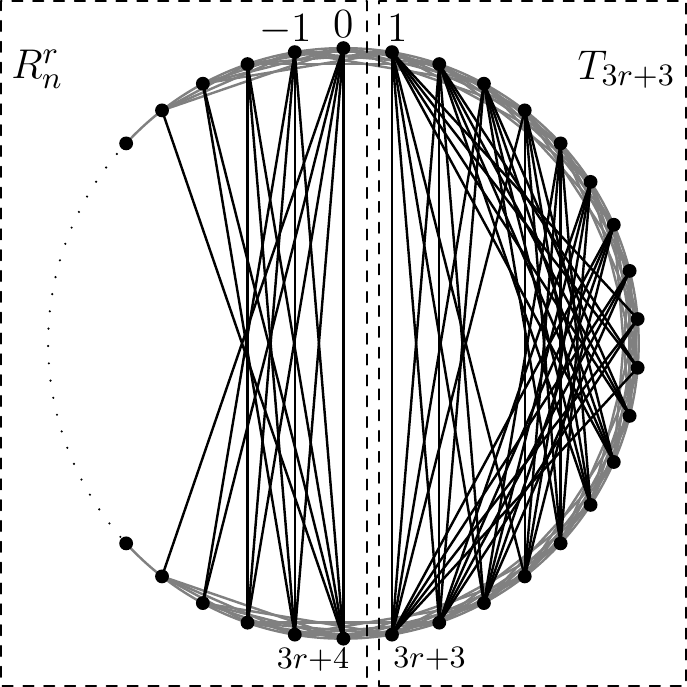}\\
\end{tabular}
\caption{The graphs $\tc$ for $r=2,3$ and a general decomposition shown for $r=5$.}
\label{fig:general}
\end{figure}

Let $\oth$ denote the subgraph of $\tc$ induced by the vertices $\{1,\ldots,3r+3\}$ (this subgraph does not depend on $n$, see Fig.\ref{fig:general}). Also, let $R_n^r$ be the remaining part of $\tc$ i.e. the subgraph induced by $\{3r+4,\ldots,-1,0\}$. Note that all the edges added in the first phase of the construction belong to $\oth$, all the edges from the second phase are in $R_n^r$ and the only edges between the two parts are those that were originally in $C_n^r$. The condition $n\geq 5r+4$ of Theorem \ref{thm:bigsplitting} guarantees that all the edges added in the construction (esp. in the second phase) are indeed ``new''.

We start with some technical properties of $\tc$ and $\oth$ which ultimately lead to the fact that $\ind(\tc)$ is a homotopical model for $\susp^2\,\ind(C_{n-(3r+3)}^r)$.

\begin{lemma}
\label{technical}
The graphs $\tc$ and $\oth$ have the following properties:
\begin{itemize}
\item[a)] The graphs $\oth$ and $\tc$ have an axis of symmetry, in the sense that there is an edge $(i,j)$ if and only if there is an edge $(3r+4-i,3r+4-j)$.
\item[b)] For any $1\leq i\leq r$ the graph $\oth\setminus N[i]$ is isomorphic to the path $P_4$ on the vertices $i+r+1$, $i+r+2$, $i+2r+2$, $i+2r+3$.
\item[c)] If $i,j$ are two vertices of $\oth$ with $1\leq j-i\leq 2r+1$ then $(i,j)$ is not an edge of $\oth$ if and only if $j=i+r+1$ or $j=i+r+2$.
\item[d)] For any $0\leq k\leq r+1$ we have $\ind(\oth[k+1,\ldots,k+2r+2])\htpyequiv\ast$.
\item[e)] There is a homotopy equivalence $\ind(\oth)\htpyequiv S^1$.
\item[f)] The graphs $R_n^r$ and $C_{n-(3r+3)}^r$ are isomorphic.
\end{itemize}
\end{lemma}
\begin{proof}
\textbf{a)} The statement obviously holds for the original edges of $C_n^r$. If $(i,j)=(i,i+2r-s+2)$ is an edge added in the $s$-th stage then
$$(3r+4-j,3r+4-i)=(r+s+2-i,(r+s+2-i)+2r-s+2)$$
was also added in the same stage as $1\leq r+s+2-i\leq r+s+1$. A similar argument applies to the edges of the form $(i,i+3r-s+3)$. Every edge $(-x,3r+3+y)$ of the second phase is mirrored by 
$$(3r+4-(3r+3+y),3r+4-(-x))=(-(y-1),3r+3+(x+1))$$
which was also added in the second phase.

\textbf{b)} Any vertex $i$ with $1\leq i\leq r$ is connected to
\begin{itemize}
\item all of $1,\ldots,i+r$ --- using the original edges from $C_n^r$,
\item vertices between $i+2r-1+2=i+2r+1$ and $i+2r-(r-1)+2=i+r+3$ (going backwards) --- edges added in the first groups of each stage as $i\leq r+s+1$ for all $s$,
\item vertices between $i+3r-i+3=3r+3$ and $i+3r-(r-1)+3=i+2r+4$ (going backwards) --- edges added in the second groups of each stage $s$ that satisfies $i\leq s$.
\end{itemize}
It means that the vertices of $\oth\setminus N[i]$ are exactly $\{i+r+1, i+r+2, i+2r+2, i+2r+3\}$. Moreover, any two of them with difference other than $1$ or $r$ have difference $r+1$ and $r+2$. Such pairs do not form edges because the shortest edges added in the first phase span over a distance of at least $2r-(r-1)+2=r+3$. That means $\oth\setminus N[i]$ is precisely a $P_4$.

\textbf{c)} As observed in b), there are no edges $(i,i+r+1)$ and $(i,i+r+2)$. If $j-i\leq r$ then $(i,j)$ is an edge already in $C_n^r$. Now suppose that $r+3\leq j-i\leq 2r+1$ and let $s=2r+i-j+2$. The constraints on $i,j$ are equivalent to $1\leq s\leq r-1$ and the inequality $j\leq 3r+3$ is equivalent to $i\leq r+s+1$. It means that the edge $(i,i+2r-s+2)=(i,j)$ was added in stage $s$.

\textbf{d)} We will show that the complement of the graph $\oth[k+1,\ldots,k+2r+2]$ is a path and then the result immediately follows. Part c) gives a complete description of edges in that complement. Vertices $k+r+1$ and $k+r+2$ have one incident edge each (to $k+2r+2$ and $k+1$, respectively) and every vertex $k+i$ with $1\leq i\leq r$ has edges to $k+i+r+1$ and $k+i+r+2$. This easily implies that the graph in question is the path
$$k+r+1,k+2r+2,k+r,k+2r+1,k+r-1,\ldots,k+1,k+r+2.$$

\textbf{e)} By part b) the complexes $\ind(\oth\setminus N[i])$ are contractible and contained in $\{r+1,\ldots,3r+3\}$ for all $1\leq i\leq r$. By Proposition \ref{prop:les-v}.a) we can therefore sequentially remove all those $i$ from $\oth$ without affecting the homotopy type of the independence complex. That means
$$\ind(\oth)\htpyequiv\ind(\oth[r+1,\ldots,3r+3]).$$ 
Let $H=\oth[r+1,\ldots,3r+3]$. Using part d) with $k=r+1$ we get that $\ind(H\setminus \{r+1\})$ is contractible. Moreover the graph $H\setminus N[r+1]$ is the $3$-vertex path induced by $2r+2,2r+3,3r+3$,  so $\ind(H\setminus N[r+1])\htpyequiv S^0$. The cofibration sequence of Proposition \ref{prop:les-v} now yields $\ind(H)\htpyequiv \susp S^0=S^1$.

\textbf{f)} This is obvious as the edges added in the second phase of the construction are exactly those needed to close the long power of a path $P_{n-(3r+3)}^r$ into the same power of a cycle.
\end{proof}

\begin{proposition}
\label{prop:model}
There is a homotopy equivalence $$\ind(\tc)\htpyequiv\susp^2\,\ind(C_{n-(3r+3)}^r).$$
\end{proposition}
\begin{proof}
We will show that all edges that connect $\oth$ with $R_n^r$ can be removed without changing the homotopy type of the independence complex, i.e. that the inclusion
$$\ind(\tc)\incl\ind(\oth\sqcup R_n^r)$$
is a homotopy equivalence. Then the result follows from e) and f) of Lemma \ref{technical}.

\begin{figure}
\includegraphics[scale=0.7]{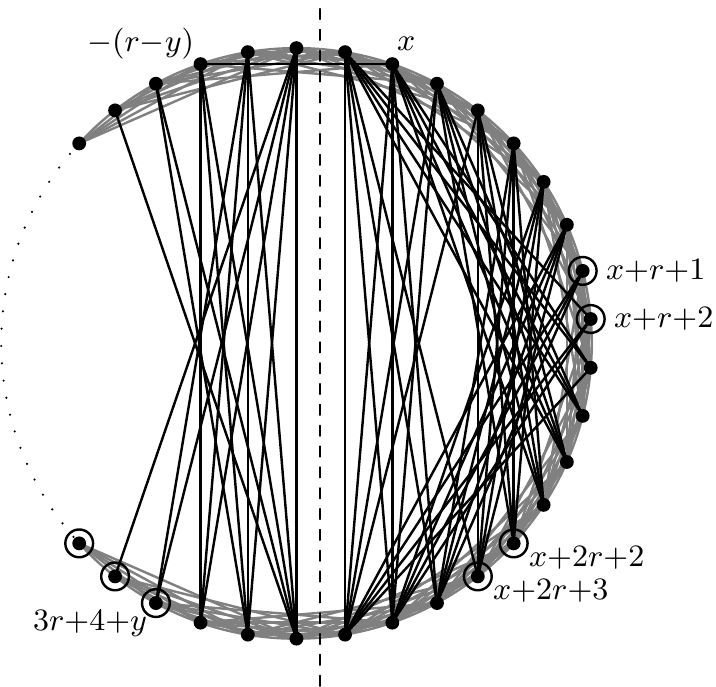} 
\caption{The proof of Proposition \ref{prop:model}. The circled vertices are those that remain after removing $N[e]$ for $e=(x,-(r-y))$.}
\label{fig:arg-x}
\end{figure}

Because of the symmetry of Lemma \ref{technical}.a) it suffices to consider the removal of edges of $C_n^r$ which ``go across $0$''. Every such edge is of the form $e=(x,-(r-y))$ for $1\le x\leq y\leq r$ (see Fig.\ref{fig:arg-x}). By Proposition \ref{prop:les-e}.a) all we need to check is that the complex $\ind(\tc\setminus N[e])$ is contractible. (To be precise, we need to know this not for $\tc$ but for the intermediate graph we obtain after some edges of this form have already been removed. It is, however, easy to see that it will be exactly the same thing.) 

By Lemma \ref{technical}.b) the removal of $N[x]$ deletes all vertices in $\{1,\ldots,3r+3\}$ except $x+r+1$, $x+r+2$, $x+2r+2$, $x+2r+3$. The removal of $N[-(r-y)]$ deletes (in particular) all of $0,\ldots,-r$ and $3r+4,\ldots,3r+3+y$. The first vertex in $R_n^r$ which remains is $3r+4+y$ and
$$(3r+4+y)-(x+2r+3)=r+1+(y-x)\geq r+1$$
so it is too far to be adjacent to the vertices which remain inside $\oth$. It follows that $\tc\setminus N[e]$ is a disjoint union of $P_4$ and some subgraph of $R_n^r$, hence its independence complex is contractible. This is what we needed to prove.
\end{proof}

We can now move on to the second part of the program outlined at the beginning of this section. This means proving:
\begin{proposition}
\label{prop:edgebyedge}
The sequence of edges listed in the construction of $\tc$ from $C_n^r$ satisfies the assumptions of Lemma \ref{lem:inductive}.
\end{proposition}
\begin{proof}
We start from $G=C_n^r$ and expand it edge by edge.

\textbf{Edges of the first phase.} Suppose we are now in stage $s$, $1\leq s\leq r-1$.

\textbf{$i$-th edge of first group.} Suppose our current graph $G$ includes all the edges up to the edge $e=(u,v)=(i,i+2r-s+2)$ of the first group in the $s$-th stage. We are going to use Theorem \ref{p4} with $x=i+1$, $y=i+2r-s+1$, see Fig.\ref{fig:arg}.a). The graph induced by $\{x,y,u,v\}$ has edges $e=uv$, $ux$ and $vy$ and no others because the differences between remaining pairs of vertices are at least $2r-s\geq r+1$ and less than $2r-s+2$, so those edges may potentially only be added in the first groups of future stages. It means that the induced graph is a $P_4$. It remains to check that $N[x]\cup N[y]\subset N[e]$.

Note that $2r-s+2\leq 2r+1$ so the whole interval $\{u-r,\ldots, v+r\}$ is in $N[e]=N[u]\cup N[v]$ already in the graph $C_n^r$. It means that all the neighbours of $x$ or $y$ in $C_n^r$ belong to $N[e]$. It remains to concentrate on the new adjacencies induced by the edges added previously in the construction. Consider first the vertex $x=i+1$. It can have the following, previously added edges.
\begin{itemize}
\item $(x,j)=(i+1,(i+1)+2r-s'+2)=(i+1,i+2r-s'+3)$ for some $1\leq s'\leq s$. Then $j-v=s-s'+1<r$, so $j\in N[v]\subset N[e]$.
\item $(j,x)=((i+1)-(2r-s'+2),i+1)=(i-2r+s'-1,i+1)$ for some $1\leq s'\leq s$ (see Fig.\ref{fig:arg}.a)). Let $s''=s'+s-r$. The inequality $j\geq 1$ is equivalent to $2r-s'+2\leq i$. Together with the inequality $i\leq r+s+1$, which holds because we are currently in stage $s$, they yield $s''=s+s'-r\geq 1$. Clearly $s''<s'$ so $s''$ is a valid number of a past stage. We also have
$$j=i-2r+s'-1\leq r+s+1-2r+s'-1=s'+s-r=s''$$
which means that in stage $s''$ we added an edge of the second group
$$(j,j+3r-s''+3)=(j,i-2r+s'-1+3r-(s'+s-r)+3)=(j,i+2r-s+2)=(j,v)$$
so $j\in N[v]\subset N[e]$ at the present stage, as required.
\item $(x,j)=(i+1,(i+1)+3r-s'+3)$ for some $1\leq s'<s$. If that edge was added in stage $s'$, we must have had $i+1\leq s'$. Let $s''=s'-1$. Then
$1\leq i\leq s''<s\leq  r-1$ so stage $s''$ existed and in that stage we added the edge
$$(i,i+3r-s''+3)=(i,i+3r-s'+1+3)=(i,j)$$
so $j\in N[u]\subset N[e]$.
\item $(j,x)=((i+1)-(3r-s'+3),i+1)=(i-3r+s'-2,i+1)$ for some $1\leq s'<s$. We must have $j\geq 1$ and $i\leq r+s+1$, so
$$1\leq i-3r+s'-2\leq r+s+1-3r+s'-2=s+s'-2r-1<0,$$
which is a contradiction.
\end{itemize}
This completes the proof that $N[x]\subset N[e]$. Note that in this proof we only used the existence of edges from previous stages and never needed to refer to the edges added earlier in the same $s$-th stage. The part of the graph constructed up to the complete $(s-1)$ stages has the axis of symmetry of Lemma \ref{technical}.a), therefore the same proof will work to show $N[y]\subset N[e]$. It means that $N[x]\cup N[y]\subset N[e]$ and the assumptions of Theorem \ref{p4} are satisfied.

\textbf{$i$-th edge of the second group.} Now suppose we are adding the edge $e=(u,v)=(i,i+3r-s+3)$ in the second group of stage $s$, and all the previous edges are already in the graph. We are going to use Theorem \ref{p4} with $x=i+1$, $y=i+3r-s+2$, see Fig.\ref{fig:arg}.b). The graph induced by $\{x,y,u,v\}$ has edges $e=uv$, $ux$ and $vy$. There are no other edges because the remaining differences are smaller than the one between $u$ and $v$, but at least $3r-s+1\geq 2r+2$, so those edges may potentially only be added in the second groups of future stages. It means that the induced graph is a $P_4$. As before, to check  $N[x]\subset N[e]$ we only need to restrict to those edges from $x$ whose endpoints are not obviously covered by the neighbours of $u$ and $v$ from $C_n^r$. Those include:
\begin{itemize}
\item $(x,j)=(i+1,(i+1)+r)$, see Fig.\ref{fig:arg}.b). Note that by construction we must have $i\leq s$ therefore $j\leq r+s+1$, so in the first group of the present stage we added the edge
$$(j,j+2r-s+2)=(j,i+1+r+2r-s+2)=(j,i+3r-s+3)=(j,v)$$
so $j\in N[v]\subset N[e]$.
\item $(x,j)=(i+1,(i+1)+3r-s'+3)$ for some $1\leq s'\leq s$. If that edge was added in stage $s'$, we must have had $i+1\leq s'$. Let $s''=s'-1$. Then
$1\leq i\leq s''<s\leq  r-1$ so stage $s''$ existed and in that stage we added the edge
$$(i,i+3r-s''+3)=(i,i+3r-s'+1+3)=(i,j)$$
so $j\in N[u]\subset N[e]$.
\item $(x,j)=(i+1,(i+1)+2r-s'+2)=(i+1,i+2r-s'+3)$ for some $1\leq s'\leq s$. But then
$$v-j=(i+3r-s+3)-(i+2r-s'+3)=r-(s-s')\leq r$$
so $j\in N[v]$ already in $C_n^r$.
\end{itemize}
It proves that $N[x]\subset N[e]$ and $N[y]\subset N[e]$ follows from symmetry as before. Again, we invoke Theorem \ref{p4} to verify the assumption in Lemma \ref{lem:inductive}.

\begin{figure}
\begin{tabular}{ccc}
\includegraphics[scale=0.7]{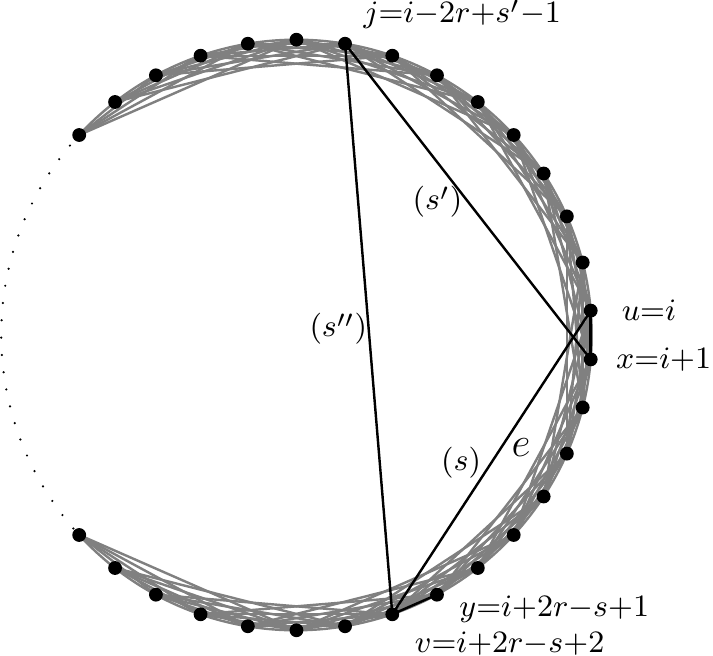} & \includegraphics[scale=0.7]{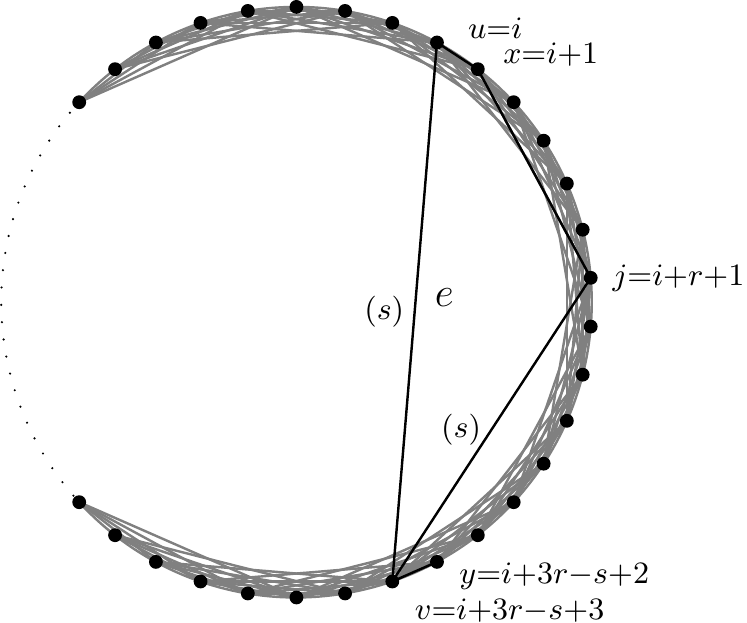} &
\includegraphics[scale=0.7]{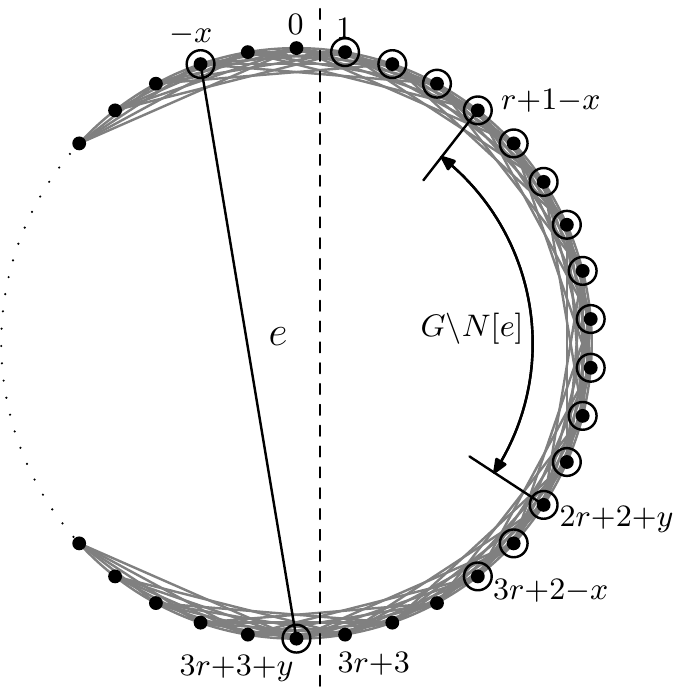}\\
a) & b) & c)
\end{tabular}
\caption{Illustration to some arguments in Proposition \ref{prop:edgebyedge}. The labels $(s)$ indicate in which stage the edge was added. In c) the circled vertices are those of the set $V$.}
\label{fig:arg}
\end{figure}

\textbf{Second phase.} Assuming that the first phase is complete we are now going to add edges of the second phase. Here the order is irrelevant to the argument. Suppose we have already constructed some graph $G$, which includes all edges of the first phase (in particular, the whole $\oth$ is already there), and that we are now adding the edge (see Fig.\ref{fig:arg}.c))
$$e=(-x,3r+3+y),\quad\mathrm{for}\quad 0\leq x\leq r-1,\, 1\leq y\leq r,\, x+y\leq r.$$
Let 
\begin{align*}
V&=\{-x\}\cup\{1,\ldots,-x+(3r+2)\}\cup\{3r+3+y\}\\
W&=\{3r+3+y+(r+1),\ldots,-x-(r+1)\}.
\end{align*}
The inclusion $\ind(e\sqcup(G\setminus N[e]))\incl\ind(G)$, which we need to show is null-homotopic, factors through $\ind(G[V\cup W])$. Indeed, $W$ contains all the vertices of $(G\setminus N[e])\cap R_n^r$. To see that $V$ covers all of $(G\setminus N[e])\cap \oth$ note that the last vertex not in $N[3r+3+y]$ is $2r+2+y$, but
$$2r+2+y\leq 2r+2+(r-x)=-x+(3r+2)$$
so $(G\setminus N[e])\cap \oth\subset V$.

There are no edges from $V$ to $W$, so $\ind(G[V\cup W])=\ind(G[V])\ast\ind(G[W])$. We are going to show that
$$\ind(G[V])\htpyequiv\ast$$
and this gives the desired conclusion.

To analyze $G[V]$ first look at the vertex $-x$. We have
$$G[V]\setminus N[-x]=\oth[-x+r+1,\ldots,-x+3r+2]$$
and the independence complex of the last graph is contractible by Lemma \ref{technical}.d). Therefore, by Proposition \ref{prop:les-v}.a) the removal of $-x$ preserves the homotopy type:
$$\ind(G[V])\htpyequiv \ind(G[V\setminus\{-x\}]).$$
But in the graph $G[V\setminus\{-x\}]$ the neighbourhood of $3r+3+y$ is $\{2r+3+y,\ldots,3r+2-x\}$. All those vertices are between $2r+4$ an $3r+2$, so they form a clique already in $C_n^r$. By Corollary \ref{clique-neib}
\begin{align*}
\ind(G[V\setminus\{-x\}])&\htpyequiv \bigvee_{i=2r+3+y}^{3r+2-x}\susp\,\ind(G[V\setminus \{-x\}]\setminus N[i])\\
&\htpyequiv \bigvee_{i=2r+3+y}^{3r+2-x}\susp\,\ind(\oth\setminus N[i])
\end{align*}
In the last wedge sum $i\geq 2r+4$, so each summand is contractible by Lemma \ref{technical}.b) combined with the symmetry of Lemma \ref{technical}.a). That ends the proof.
\end{proof}

So far we proved that the splitting of Theorem \ref{thm:bigsplitting} holds for some space $X_{n,r}$. Lemma \ref{lem:inductive} also provides a description of $X_{n,r}$ as a wedge sum of $\susp^2\,\ind(G_i\setminus N[e_i])$ where $e_i$ runs through the edges added in the construction of $\tc$. We will briefly sketch how to identify those summands and this will complete the proof of Theorem \ref{thm:bigsplitting}.
\begin{itemize}
\item \emph{First groups in first phase.} For each stage $s$ if $e=(i,i+2r-s+2)$ then:
\begin{itemize}
\item For every $1\leq i\leq s-1$ the removal of $N[e]$ leaves only the vertex $v=i+3r-s+3$ and a segment isomorphic to $P_{n-4r+i-4}^r$ within $R_n^r$. The vertex $v$ is adjacent to the $r-s+i$ initial vertices of the path power. They form a clique so Corollary \ref{clique-neib} identifies $\susp^2\,\ind(G\setminus N[e])$ as 
$$\susp^3\,\ind(P_{n-5r+i-5}^r)\vee\cdots\vee\susp^3\,\ind(P_{n-6r+s-4}^r)$$
and this is the contribution of each pair $(i,s)$ with $1\leq i<s\leq r-1$. 
\item For every $r+3\leq i\leq r+s+1$ the situation is symmetric, so we can just include the contribution of the previous part twice.
\item When $s\leq i\leq r+2$ then the vertices left after removing $N[e]$ form a $P_{n-4r+s-3}^r$. For every $s$ there are $r+3-s$ suitable values of $i$, so the total contribution of this part for every $s$ is
$$\bigvee^{r+3-s}\susp^2\,\ind(P_{n-4r+s-3}^r).$$
This can be expanded into third suspensions using \eqref{engstrom-path}.
\end{itemize}
\item \emph{Second groups in first phase.} For each stage $s$ if $e=(i,i+3r-s+3)$ then the removal of $N[e]$ leaves a disjoint union of $P_{n-5r+s-4}^r$ with a clique of size $r-s$ induced by $\{i+r+2,\ldots,i+2r-s+1\}$. There are $s$ edges in this group, so here stage $s$ contributes
$$\bigvee^{s(r-s-1)}\susp^3\,\ind(P_{n-5r+s-4}^r)$$
(in particular when $s=r-1$ the clique has size $1$ and the summand is contractible) .
\item \emph{Second phase.} For an edge $e=(-x,3r+3+y)$ its removal leaves a copy of $P_{n-5r-4-(x+y)}^r$ and a segment of $\oth$ induced by $\{r+1-x,\ldots,2r+2+y\}$. The independence complex of the last piece equals
$$\ind(\oth[y+1,\ldots,y+2r+2])\setminus\{y+1,\ldots,y+(r-x-y)\}.$$
In the proof of Lemma \ref{technical}.d) we saw that $\ind(\oth[y+1,\ldots,y+2r+2])$ is homeomorphic to a path. The order of the vertices of that path implies that the removal of each of $y+1,\ldots,y+(r-x-y)$ increases the number of connected components by $1$. Therefore the resulting space is homotopy equivalent to the wedge of $r-(x+y)$ copies of $S^0$. Since the possible values of $x+y$ are $t=1,\ldots,r$ and value $t$ is attained $t$ times we get that the total contribution of the second phase is
$$\bigvee_{t=1}^r\bigvee^{t(r-t)}\susp^3\,\ind(P_{n-5r-t-4}^r)$$
(again, the summands for $r=t$ are trivial). 
\end{itemize}

A tedious calculation, which will be omitted, allows to express the combination of all the contributions in the following form.

\begin{corollary}
\label{finally}
The space $X_{n,r}$ of Theorem \ref{thm:bigsplitting} satisfies
$$X_{n,r}\htpyequiv\susp^3\,\bigvee_{i=4r+6}^{6r+3}\bigvee^{k_i}\ind(P_{n-i}^r)$$
where 
\begin{equation*}
k_i=\begin{cases}
\frac{1}{2}(i-4r-5)(i-2r-2) & \mathrm{for}\ i\leq 5r+4,\\
\frac{1}{2}(6r+4-i)(i-2r-1) & \mathrm{for}\ i\geq 5r+5.
\end{cases}
\end{equation*}
\end{corollary}

\begin{remark}
This work provides a natural recursive relation for $\ind(C_n^r)$, but does not say anything about the ``initial conditions'', that is the case when $n<5r+4$. It is reasonable to expect that all those spaces are, up to homotopy, wedges of spheres. Other methods of computing their homotopy types were recently obtained in \cite{TestaPriv,Jojic}.
\end{remark}

\subsection*{Acknowledgement} The author thanks Dmitry Kozlov, John Jones, Damiano Testa, Russ Woodroofe and the referees for their remarks. The software package \texttt{polymake} \cite{Poly} was very helpful in the formulation and verification of various conjectures about $\ind(C_n^r)$.


\end{document}